\documentclass[a4paper,oneside, 12pt]{amsart}          
\usepackage{amsfonts,amsmath,latexsym,amssymb} 
\usepackage{amsthm}      
\usepackage[dvipsnames]{xcolor}          
\usepackage{float}
\usepackage{tikz}
\usepackage{circuitikz}
\usepackage{hyperref}
\usepackage{mathtools}
\usepackage{enumitem}
\usepackage[pdftex]{lscape}
\usepackage{geometry}
\usepackage{bbm}
\usepackage[font=footnotesize,labelfont=bf]{caption}

\allowdisplaybreaks

\newtheorem{theorem}{Theorem}  
\numberwithin{theorem}{section}               
\newtheorem{lemma}[theorem]{Lemma}               
\newtheorem{corollary}[theorem]{Corollary}
\newtheorem{proposition}[theorem]{Proposition}

\theoremstyle{definition}

\newtheorem{definition}[theorem]{Definition}
\newtheorem{example}[theorem]{Example}

\newtheorem{remark}[theorem]{Remark}

\newtheorem{assumption}[theorem]{Assumption}

\renewcommand{\c}{\operatorname{cap}}

\newcommand{\G}{{G}}
\newcommand{\N}{\mathcal{N}}
\newcommand{\F}{\mathcal{F}}
\newcommand{\K}{\mathbb{K}}
\newcommand{\Kf}{\mathbb K_{\precsim 1}}
\newcommand{\map}{\varrho}
\newcommand{\prob}{\pi}
\newcommand{\Prob}{\Pi}
\newcommand{\rw}{\beta}

\newcommand{\pp}{\prec\hspace{-2pt}\prec}
\renewcommand{\ss}{\succ\hspace{-2pt}\succ}
\newcommand{\subsV}{U}

\author[M. Keller, A. Muranova]{Matthias Keller, Anna Muranova}

\address{Matthias Keller: Institut f\"ur Mathematik, Universit\"at Potsdam
	14476  Potsdam, Germany}
\email{matthias.keller@uni-potsdam.de}

 \address{Anna Muranova: Faculty of Mathematics and Computer Science, University of Warmia and Mazury in Olsztyn, ul. Sloneczna 54, 10-710 Olsztyn, Poland} 

 \email{anna.muranova@matman.uwm.edu.pl}

\title[]{Recurrence and transience for non-Archimedean and directed graphs}
\thanks{}

\begin{document}

\maketitle

\begin{abstract}
We introduce the notion of recurrence and transience for graphs over non-Archimedean ordered field. To do so we relate these graphs to random walks of directed graphs over the reals. In particular, we give a characterization of the real directed graphs which can arise is such a way. As a main result, we give characterization for recurrence and transience in terms of a quantity related to the capacity. 
\end{abstract}

 \footnotetext[1]{The first author acknowledges the financial support of the DFG.  	The second author acknowledges financial support by the NCN (National Science Center, Poland), grant nr. 2022/06/X/ST1/00910.}     

{\footnotesize
\smallskip


\section{Introduction}

The study of random walks on graphs is a classical topic going back to the last century. The name was coined by  Karl Pearson in 1905 \cite{Pearson} to  model a mosquito infestation in a forest. Since then this has become a vibrant field of mathematics connecting probability, potential theory, graph and geometric group theory, see e.g. the monographs \cite{LyonsPeres,Woess00,Woess09} and references therein as well as \cite{BosiHuPeres,Cenac,Kumamoto} for recent developments. 

In this work we take on the study of random walks arising from graphs over non-Archimedean fields. Such graphs have been recently introduced and studied in \cite{Muranova1,Muranova2,OurPreprint}. In particular, in \cite{OurPreprint} the notion of capacity on infinite graphs was studied in great depth. Classically the notion of capacity is closely related to questions of the behavior of random walks. Indeed vanishing capacity at a vertex is equivalent to fact that the random walk returns to this vertex over time infinitely often almost surely -- a property which is known as recurrence. Conversely, if the capacity is positive the random walk only returns finitely many times almost surely which is known as transience.

However, for non-Archimedean graphs the situation is significantly more subtle. First of all, while the sequence of capacities along an exhaustion of an infinite graph is still monotone decreasing its limit does not necessarily exist. So, the classical dichotomy of zero or positive capacity is extended by the case of divergent capacity. Secondly, convergence to zero is much harder to realize in non-Archimedean fields as the sequence is not only required to be eventually smaller than any positive rational but also than any infinitesimal. On the other hand, in terms of non-Archimedean probability   a property is said to hold almost surely if it holds outside of an event  which has probability less than every positive rational \cite[p.~25]{Nelson}. Thus, expecting to equate zero capacity with recurrence would be also conceptually problematic.

Thus, we have to take heed in two ways. First of all, we define random walks and the notions of recurrence and transience in accordance to the proper viewpoint of probability on non-Archimedean fields. This leads us to relate our non-Archimedean graphs to certain directed graphs over the reals. For such graphs, there is a rich theory of their corresponding random walks, see e.g. \cite{Woess09}, which we will exploit. However, it is crucial to align the proper notions in the non-Archimedean and the corresponding directed real setting. In particular, this relates to paths in the non-Archimedean graph with  directed paths and consequently with (strongly) connected components.  For directed graphs the study of connected components is more subtle and a specific importance comes to so called essential components. These are also called absorbing components and are characterized by the fact that they do not have paths leaving this component.
We use these notions to obtain a first result on determining whether a vertex is recurrent or transient. This result, Theorem~\ref{thm:rectranscomponent}, says that vertices in non-essential components  are necessarily transient. This basically leaves us with the study of the type of the vertex in the essential components of the graph. Furthermore,  we characterize the directed graphs over the reals which arise from non-Archimedean graphs in Theorem~\ref{thm:infinitevsdirected}.

Secondly, we consider a quantity which relates the capacity to recurrence and transience. For real graphs the diagonal of the Green's function is equal to the reciprocal of the capacity normalized by the measure. As discussed above neither of these quantities necessarily exists for infinite non-Archimedean graphs. However, as the normalized capacities  of a vertex $ a $ along an exhaustion are bounded by values less than one, we can take their real part -- that is the unique real number in the field which differs only by an infinitesimal. This gives not only a real number between zero and one but  also  the limit exists within the reals. Hence, also its reciprocal called $ \G(a) $ exists and is finite or positive infinite. For vertices in the essential components this allows us to characterize recurrence and transience in terms of infiniteness or finiteness of $ \G(a) $,  Theorem~\ref{thm:charGesscomp}. See also Corollary~\ref{cor:charGesscomp} which brings this together with the findings of Theorem~\ref{thm:rectranscomponent} and therefore provides a full characterization of recurrence and transience. Moreover, for vertices in any non-essential component,  one direction of this characterization remains true, Corollary~\ref{cor::rec}, namely that $ \G(a) <\infty$ implies transience of $ a $. Of course, this raises the question whether a full characterization of recurrence and transience is possible in terms of $ \G(a) $. That this is not the case  is shown in Examples~\ref{ex1} and Example~\ref{ex2}.

The paper is structured as follows. In the next section we provide all the basic notions and facts for graphs over non-Archimedean fields first and random walks over directed graphs over the reals second. In Section~\ref{sec:Correspondence} we then construct directed real graphs from non-Archimedean ones together with their random walks. This also allows us to define recurrence and transience for vertices in the original graph.  We then give a first criterion for recurrence and transience, Theorem~\ref{thm:rectranscomponent} and also characterize the directed real graphs which can be obtained from non-Archimedean ones, Theorem~\ref{thm:infinitevsdirected}. In Section~\ref{sec:G}, we then introduce the quantity $ \G(a) $ for a vertex $ a $ discussed above and study its basic properties. Finally, in Section~\ref{sec:char} we characterize recurrence and transience of vertices in the graph with the help of $ \G(a) $ in Theorem~\ref{thm:charGesscomp} and Corollary~\ref{cor:charGesscomp}.

\section{Set up and preliminaries}
\subsection{Weighted graphs over ordered field}
An ordered field $(\K, \succ)$ is called {\em non-Archimedean} if there exists an {\em infinitesimal}, i.e., $\tau \succ 0$  in  $\Bbb K$ such that 
$$
{\tau}\prec \dfrac{1}{n}=\dfrac{1}{\underbrace{1+\dots+1}_{n\mbox{ \tiny{times}}}}
$$
for any $n\in \mathbb N\subseteq\K$. Otherwise, the field is called {\em Archimedean}. Any element $\mathcal N$ with $\mathcal N\succ N$ for all $N\in \Bbb N$ is called {\em infinitely large}. Obviously, the field is non-Archimedean if and only if there exists an infinitely large element.

We denote the positive elements of $ \K $ by $ \K^+=\{x\in \K\mid x\succ 0\} $. The \emph{absolute value $|k|$  of $k\in \K$} is an element of  $ \K^+\cup\{0\} $ defined as 
\begin{equation*}
|k|=\begin{cases}
k,&:\mbox{if }k\succeq 0\\
-k&: \mbox{otherwise.}
\end{cases}
\end{equation*}

Clearly, also the  rationals $ \mathbb{Q} $ are a subfield of any ordered field. For two elements $k_1,k_2\in \Bbb K^+$ we fix the following notation:\\
We write
$k_1\pp k_2$ if  $\dfrac{k_1}{k_2}$ is an infinitesimal, i.e.,
$$
k_1\pp k_2\qquad:\Longleftrightarrow\qquad\dfrac{k_1}{k_2}\prec c\qquad\mbox{for all }c\in \Bbb Q^+.
$$
We write
$k_1\precsim k_2$ if    $\dfrac{k_1}{k_2}$ is at most finite, i.e.,
$$
k_1\precsim k_2\qquad:\Longleftrightarrow\qquad \dfrac{k_1}{k_2}\prec c_0\qquad\mbox{for some }c_0\in \Bbb Q^+.
$$
We write
 $k_1\simeq k_2$, if $k_1\precsim k_2$ and $k_1\succsim k_2$, i.e.
$$
k_1\simeq k_2\qquad:\Longleftrightarrow\qquad c_1\prec \dfrac{k_1}{k_2}\prec c_2 \qquad\mbox{for some }c_1, c_2 \in \Bbb Q^+
$$
Note that $\simeq $ is an equivalence relation.  Further, we say that $k\in \Bbb K$ is \emph{at most finite}, if $k\precsim 1$, i.e., $k$ is not infinitely large.  We denote by $\Kf$ the set of all at most finite elements in the field $\K$.

\begin{definition}
Let $(\K, \succ)$ be an ordered field.  {\em A weighted graph} over  an at most countable  set $V$, whose elements are called the \emph{vertices},  is a symmetric map
$b:V\times V\rightarrow \K^+\cup\{0\}$  with vanishing diagonal, (i.e., $ b(x,y)=b(y,x)  $ and $ b(x,x)=0 $ for all $ x,y\in V $).
\end{definition}

Having zero diagonal  means that that graphs have no {loops}. We write $x\sim y$ whenever $b(x,y)\succ 0$ and say that there is an {\em edge} between $x$ and $y$.

An {\em (undirected) path} between any vertices $x,y\in V$ is a sequence $(x_1,\ldots, x_{n}), n\in \mathbb N$ such that
$$
x=x_0\sim x_1\sim x_2\sim \dots \sim x_n=y.
$$
A graph is called {\em connected}, if there is a path between any two vertices.  
Furthermore, a subset $ U\subseteq V $ is called \emph{connected} if the restriction $ b\left|_{U\times U}\right.$  is connected.
 A graph is called {\em locally finite}, if for all $x\in V$ we have $$ \#\{y\in V\mid y\sim x\}<\infty .$$  

\begin{assumption}
	In the following, $ b $ always denotes a connected locally finite graph over an at most countable set $ V $ taking values in an ordered field $ \mathbb{K} $.
\end{assumption}
For any $U\subseteq V$ and $ x\in V $, we denote
$$
b(U)=\sum_{z\in U}\sum_{y\in V}b(z,y) \qquad \mbox{ and }\qquad b(x):=b(\{x\})=\sum_{x\in V}b(x,y).
$$
The \emph{normalized weight} $p:V\times V\to \mathbb{K}$ of a graph $ b $ over $ V $ is defined as
\begin{equation*}
p(x,y)=\dfrac{b(x,y)}{b(x)}.
\end{equation*}
Moreover, for $ U\subseteq V $, we define the following set of functions:
$$
\F(U)=\{f \mid f:U\to \K\}
$$
and
$$
C_c(V)=\{\phi \mid \phi:V\to \K, \phi\mbox{ has finite support}\}
$$
where the support of the function is the set of vertices where this function does not vanish. For the compactly supported functions $ C_{c}(V) $, the
characteristic functions $ 1_{x} $ of vertices $ x\in V $ form a basis.
We define a\emph{ transition operator}  by
$$
P:\F(V)\to\F(V),\qquad Pf(x)=\sum_{y\in V} f(y)p(x,y).
$$
By local finiteness,  we can define $P^n$  for  $n\in \mathbb N$ inductively and let $P^0=I$. The matrix elements of $P^n$ with respect to the standard basis are denoted by $p^{(n)}(x,y)$. Note that, for $ n\ge 2 $,
$$
p^{(n)}(x,y) = \sum_{\{x_1,\ldots,x_{n-1}\}\subseteq V}p(x, x_1)p(x_1,x_2)\dots p(x_{n-1},y).
$$

The capacity can be defined via limit  involving solutions of Dirichlet problems or via the  infimum of energies over compactly supported functions. The equivalence of either definition  is shown in \cite{OurPreprint} and we present here the main concepts.

The \emph{Laplace operator} is defined on the set of all functions $ \mathcal{F}(V) $
by
\begin{equation*}
\Delta f(x)=\dfrac{1}{b(x)}\sum_{y\in V}b(x,y)(f(x)-f(y)).
\end{equation*}
and the corresponding quadratic form or \emph{energy form} for $ \phi \in C_{c}(V)$ is defined as
\begin{equation*}
	Q(\phi)=\dfrac{1}{2}\sum_{x,y\in V}b(x,y)(\phi(x)-\phi(y))^2,
\end{equation*}
which is connected to $ \Delta $ via \emph{Green's formula} for $ f\in \mathcal{F} (V)$ and $ \phi\in C_{c}(X) $, cf. \cite[Theorem~22]{Muranova1}
$$
\sum_{x\in V}(b\phi\Delta f)(x)  =\sum_{x\in V} (bf\Delta \phi)(x) =\dfrac{1}{2}\sum_{x,y\in V}b(x,y)(f(x)-f(y))(\phi(x)-\phi(y))=:Q(f,\phi).
$$

Let $ \emptyset\neq K\subseteq V $ be  a finite  connected subset of $V$ and $ a\in K $. The Dirichlet problem for $ K $ and $ a $ then looks for a solution $ v: V\to \mathbb{K} $ of the equation
\begin{equation}\label{dirpr} \tag{DP}
\begin{cases}
\Delta v=0 &\mbox{ on  }K\setminus\{a\},\\
v=1&\mbox{ on  }\{a\},\\
v= 0 &\mbox{ on  }V\setminus K.
\end{cases}
\end{equation}

\begin{proposition}[Maximum principle \cite{Muranova1}]\label{thm::maxpr} Let $ K\subseteq V $ be finite and connected and $ a\in K $. 
	 The Dirichlet problem \eqref{dirpr} has a unique solution $ v $ which   satisfies
	 $$  0 \prec v \preceq 1\qquad \mbox{on } K$$
	 and
	 \begin{equation*}
	Q(v)=\Delta v(a)b(a)=-\sum_{x\in V\setminus K}\Delta v(x)b(x).
	 \end{equation*}
	 Moreover, for any function $f\in C_c(V)$ with $f(a)=1$  and $f=0$ on $ V\setminus K $,
	 \begin{equation*}
	 Q(v)\preceq Q(f),
	 \end{equation*}
	 i.e., the solution of the Dirichlet problem minimizes the energy.
\end{proposition}

For a set $ U \subseteq V $, we denote by $ \pi_{U}: \mathcal{F}(V)\to \mathcal{F}(U) $ the canonical projection and by $ \iota_U: \mathcal{F}(V)\to \mathcal{F}(U)  $ the embedding via extension by zero. For an operator $ A: \mathcal{F}(V)\to \mathcal{F}(V)  $, we then denote by $ A_{U} =  \pi_U A\iota_{U} $ the corresponding operator on $ \mathcal{F}(U) $. 
The matrix elements $ p^{(n)}_\subsV(x,y) $, $x,y\in \subsV$ of  powers of the restricted transition operator $ P_{U}^{n} $ satisfy $ p^{(1)}_{\subsV} (x,y)=p_{U}(x,y) $ and, for $ n\ge 2 $,
$$ p^{(n)}_\subsV(x,y) = \sum_{\{x_1,\ldots,x_{n-1}\}\subseteq\subsV}p(x, x_1)p(x_1,x_2)\dots p(x_{n-1},y).$$
However, observe that one does not longer have $ \sum_{y\in U}p({x,y})=1 $ for $ x\in U $ but only $ \sum_{y\in U}p({x,y})\preceq 1 $ which is a strict inequality whenever there is $ y\in V\setminus U $ with $ x\sim y $.

For   $K\subseteq V$ finite and  $a\in K$, the {\em renormalized Dirichlet problem} looks for $ \widetilde{v}\in \mathcal{F}(K) $ such that
\begin{equation}\label{dirprmod} \tag{$ \widetilde{\mathrm{DP}}$}
\Delta_K\widetilde v= 1_a
\end{equation}
and we call $ \Delta_{K} $ the \emph{Dirichlet Laplacian for $ K $}.
There is the following connection to the Dirichlet Problem \eqref{dirpr}.

\begin{lemma}[Renormalized Dirichlet problem, Lemma 2.5 and~4.5 in \cite{OurPreprint}]\label{lem:RenDP}
Let   $K\subseteq V $ be finite and connected and $ a\in K $. 
	The Dirichlet problem \eqref{dirprmod} has a unique solution $\widetilde v $  which satisfies
	$$ \widetilde v=\dfrac {v}{\Delta v(a)}\qquad\mbox{and }\qquad v=\dfrac{ \widetilde v}{\widetilde v(a)}, $$ where $v$ is the solution of the Dirichlet problem \eqref{dirpr}. Furthermore, the Dirichlet Laplacian $ \Delta_{K} $ is invertible   and 
	$$
	\widetilde v=\Delta^{-1}_{K} 1_a\succeq 0.
	$$
\end{lemma}

\begin{definition}[Capacity, Definition~2.3 \cite{OurPreprint}]\label{def::capK}
The \emph{capacity}  $ \c_{K} $
for a finite and connected set $K\subseteq  V$ and $ a\in K $ is defined as
$$
\c_{K}(a)=\Delta v(a)b(a),
$$
where $v$ is the solution of the Dirichlet problem \eqref{dirpr}.
\end{definition}

From the definition, Lemma~\ref{lem:RenDP} and Green's formula, we can  immediately conclude at
\begin{equation*}
\c_{K}(a)=\dfrac{b(a)}{\Delta^{-1}_{K} 1_a(a)}=Q (v).
\end{equation*}

\begin{lemma}[Lemma 2.6 in \cite{OurPreprint}]\label{lem::FiniteSolxy} Let $ K \subseteq V$ be finite and connected and $ x,y \in K $. Let $ v^{x} $, respectively $ v^{y} $, be solutions of \eqref{dirpr} for $ K $ and $x  $, respectively $ y $. Then,
$$
\dfrac{v^x(y)}{\c_{K}(x)}=\dfrac{v^y(x)}{\c_{K}(y)}.
$$
\end{lemma}

\begin{lemma}[Monotonicity of capacity]\label{lem::monOfcap}
If $a\in K\subseteq L\subseteq V$, and $K, L$ are finite, then
$$
\c_{L}(a)\preceq \c_{K}(a).
$$
\end{lemma}

\begin{proof}
Follows from the maximum principle, Proposition \ref{thm::maxpr} as this shows
$
Q(v_L)\preceq Q(v_K),
$
where $v_L$ and $v_K$ are the solutions of Dirichlet problem \eqref{dirpr} on $L$ and on $K$ respectively.
\end{proof}

We call an increasing sequence of finite connected subsets $ K_{n}\subseteq V $ an \emph{exhaustion} of $V$ if $V=~ \bigcup_{n}K_{n} $.

\begin{definition}[Effective Capacity, Definition 17 in \cite{Muranova2}]\label{defRT} 
Let $(K_n)$ be an exhaustion of $ V $. If for a fixed vertex $a\in K_0\subseteq V$ the limit of the sequence $(\c_{K_n}(a))$ exists in the order topology of $\K$, we call it \emph{effective capacity of $a$} and denote it by 
$$ 
\c(a) :=\lim_{n\to\infty }\c_{K_n}(a).
$$
\end{definition}

The following theorem shows that existence of the limit is independent of the base vertex $ a $.

\begin{theorem}[Theorem 2.9 in \cite{OurPreprint}]\label{Thm::ind}
 If $ \c(a) $ exists for some $ a\in V $, then it exists for all $ x\in V $. Moreover, if $ \c(a) =0$  for some $ a\in V $, then $ \c(x) =0$  for all $ x\in V $.
\end{theorem}

The previous theorem gives rise to the following definition of the capacity type of a graph.

\begin{definition}[Definition~2.10 in \cite{OurPreprint}]\label{Def::type}
The  graph   is said to have
\begin{enumerate}
\item  
{\em null capacity}, if $\c(a)=0$ for some (all) $a\in V$,
\item 
{\em positive capacity}, if $\c(a)$ exists and $\c(a)\ne 0$ for some (all) $a\in V$,
\item
{\em divergent capacity}, if $\c(a)$ does not exist for some (all) $a\in V$.
\end{enumerate}
\end{definition}

\subsection{Random walks and directed graphs over the reals}
In this subsection we recall the basic notions of random walks over the Archimedean field of the reals. We will use this theory later on to connect it to graphs over non-Archimedean fields.

\begin{definition}[Transition matrix]
 A \emph{transition matrix} over an at most countable set $ V $, called the \emph{state space}, is  a map
$\prob:V\times V\rightarrow \mathbb R^+\cup\{0\}$ such that 
$$
\sum_{y\in V}\prob(x,y)= 1
$$ 
for any $x\in V$.
\end{definition}
We denote the corresponding operator on the space of compactly supported functions $ V\to \mathbb{R} $ by  $\Prob$.  For $ \Prob $, we can take again powers and we denote the matrix elements of $ \Prob^{n} $ by $ \prob^{(n)}(x,y) $ for $ x,y\in V $.  
A transition probability $ \prob $ gives rise to a  random walk $ X=(X_{n}) $ via
\begin{align*}
	\mathbb{P}( X_{n}=y \mid  {X}_{0}=x)=\prob^{(n)}(x,y).
\end{align*}
Note that we can  consider $ \pi $ as a $\emph{directed graph}$ over $ V $ which is said to be \emph{locally finite} if $\#\{y\in V\mid \pi(x,y) \} <\infty$ for all $ x\in V $. We write $x\leadsto y$ if $\prob(x,y)\ne 0$. We write $x\leftrightsquigarrow y$ if $x\leadsto y$  and $y\leadsto x$.

Further,  we write $x\to y$, i.e., \emph{$y$ can be reached from $x$}, if $ x=y $ or there exists a {\em directed path} from $x$ to $y$, i.e., there exists  $(x_i)_{i=0}^n\subseteq V, n\in \mathbb N$ such that
\begin{equation*}
x=x_0\leadsto x_1\leadsto x_2\leadsto \dots \leadsto x_n=y.
\end{equation*}
We write $x\leftrightarrow y$ if $x\to y$ and $y\to x$.
Clearly, 
$\leftrightarrow $ is an equivalence relation.
We call an equivalence class an {\em irreducible class}. In graph theoretical terminology, one also speaks of a \emph{strongly connected component} \cite[p.28]{Woess09}.

We extend the  use symbol $\to$ to irreducible classes as follows: for two irreducible classes $C_1, C_2\subseteq V$, we write $C_1\to C_2$ if for some (all) $x\in C_1$ and for some (all) $y\in C_2$ we have  $x\to y$. The following classical fact holds, see e. g. \cite[2.4 Lemma]{Woess09}:

\begin{lemma}
The set of all irreducible classes  is a partially ordered set with respect to the relation $\to$.
\end{lemma}

\begin{definition}[Essential classes]
A maximal element in the set of irreducible classes  with respect to order $\to$ is called \emph{essential class} (or \emph{absorbing class}). 
\end{definition}

\begin{lemma}
Let  $\prob$ be a transition matrix over $V$. A state $x\in V$  is in an essential class if and only if  $x\to y$ implies $y\to x$ for any $y\in V$.
\end{lemma}
\begin{proof}
	This is clear from the definition.
\end{proof}

\begin{definition}[Closed sets and absorbing states]
A subset $W\subseteq V$ is called {\em closed}, if no element outside $W$ can be reached from any state $x\in W$. A single state $x$, forming a closed set, is called {\em absorbing}.
\end{definition}

\begin{assumption}
We  assume that $\prob(x,x)>0$ if and only if $\prob(x,x)=1$, i.e., we do not have loops, except for absorbing states.
\end{assumption}

The definition of recurrence and transience is classical, see e.g. \cite[Chapter 3]{Woess09}, i.e., 
for a random walk $X=(X_n) $ on a state space $V$ with transition matrix $\prob$, a state $x\in V$ is called \emph{recurrent} if
$$
\mathbb P[X_n=x\mbox{ for infinitely many }n]=1
$$
and \emph{transient} otherwise. In fact, in case of transience we have 
$$
\mathbb P[X_n=x\mbox{ for infinitely many }n]=0,
$$
see e.g. \cite[3.2. Theorem]{Woess09}.

The classical theory gives the following connection between the type of a state and convergence of sums over $ \pi^{(n)} $, see e.g. \cite[3.4. Theorem (a)]{Woess09}
\begin{theorem}\label{thm:defrectr} Let $\prob$ be a transition matrix over $ V $. Then  the following holds:
\begin{enumerate}[label=$(\arabic*)$]
\item[(a)]
A state $x\in V$ is transient if and only if 
$$\displaystyle\sum_{n=0}^{\infty} \prob^{(n)}(x,x)<\infty.
$$
\item[(b)]
A state $x\in V$ is recurrent if and only if 
$$
\displaystyle\sum_{n=0}^{\infty} \prob^{(n)}(x,x)=\infty.
$$
\end{enumerate}
\end{theorem}

It is known, that  being recurrent or transient is a property of an irreducible class, see e.g. \cite[Theorem 3.4 and p. 45]{Woess09}
\begin{theorem}
All the vertices of the same class are recurrent or transient simultaneously.
\end{theorem}

Therefore, the type -- recurrent or transient -- is defined for irreducible classes. The following theorem is found in \cite[3.4. Theorem (b) and (c)]{Woess09}

\begin{theorem}\label{lem:recvsess}
 Every
recurrent class is an essential class and 
every finite essential class is recurrent.
\end{theorem}

\begin{corollary}
Let  $ V $ be finite and $\prob$ be a transition matrix over $V$. A state $x\in V$  is recurrent if and only if  $x\to y$ implies $y\to x$ for any $y\in V$.
\end{corollary}
\begin{proof} Clearly, if $ x $	is recurrent, then there must be a path to return to $ x $ from any $ y $ which can be reached by $ x $. On the other hand,  if $x\to y$ implies $y\to x$ for all $ y\in V  $, then $x$ is in an essential class.
\end{proof}

\section{A correspondence between  non-Archimedean graphs and  random walks}\label{sec:Correspondence}
It is known that real graphs stand in a one-to-one relationship to reversible random walks. In this section we identify a class of random walks which stand in a one-to-one correspondence to graphs over non-Archimedean fields.  In turn many of the investigations for the probability theory on non-Archimedean graphs can draw from the rich theory of those random walks.

\subsection{Constructing a  random walk for a non-Archimedean graph}

In this subsection we construct a mapping between a non-Archimedean graph and a random walk. To this end recall the set $ \Kf $ of at most finite elements which is an ordered ring and which we will map to $ \mathbb{R} $ as described below. To this end, we are looking for the  unique real number, from which an element $ k\in\K $ differs by an infinitesimal. Note that $ \K $ includes $ \mathbb{Q} $ as $ \K $ is an ordered field.

\begin{lemma}\label{lem:existencerho}
	For any $k\in \Kf$, there is a unique real number $ \rho(k) $, which differs from $k$ only by an infinitesimal in an ordered field extension including $ \K $ and $ \mathbb{R} $.
\end{lemma}
\begin{proof}
	Let $ \mathbb{F} $ be an ordered extension which includes $ \mathbb{K} $ and $ \mathbb{R} $, e.g. one can take the surreal numbers, which contains all non-Archimedean fields as subfields  (see \cite[Theorem 24.29]{Bajnok}).
	For $ k\in \Kf $, let
	\begin{align*}
		L(k)=\{  q\in \mathbb{Q}\mid q\preceq k\}\qquad\mbox{and}\qquad 	U(k)=\{  q\in \mathbb{Q}\mid q\succeq k\}
	\end{align*}
Then, $ L(k) $ and $ U(k) $ are also subsets of $ \mathbb{R} $ in which they have a supremum $ l=\sup_{\mathbb{R}}L(k) $ and an infimum $ u=\inf_{\mathbb{R}}U(k) $.
Since $ l $ and $u  $ both exist in $ \mathbb{F}$,  they differ from $ k $ only by an infinitesimal and are therefore equal in $ \mathbb{R} $. Then, this is the number $ \rho(k) $ we are looking for. Uniqueness is clear by  the triangle inequality in $ \mathbb{F} $.
\end{proof}

\begin{definition}[Real part]\label{def::map}
 We define the map  $\map:\Kf\to\Bbb R$ so that  $\map(k)$ is the unique real number  such that $k-\map(k)$ is infinitesimal in an ordered field extension including $ \K $ and $ \mathbb{R} $.
\end{definition}

We collect the following basic properties of the map $ \rho $.
\begin{lemma}[Basic properties of $ \rho $]\label{lem::mapProp}
	The map $ \rho $ is an order preserving  ring morphism, i.e., for  $ k_{1},k_{2}\in \Kf $, we have 
 $\map(k_1+k_2)=\map(k_1)+\map(k_2)$,  $\map(k_1\cdot k_2)=\map(k_1)\cdot \map(k_2)$ and  if $k_1\succeq k_2$, then $\map(k_1)\ge \map(k_2)$.
Furthermore, we have the following properties for  $ k_{1},k_{2}\in \Kf $.
\begin{enumerate}[label=$(\arabic*)$]
\item
$ \map $ maps infinitesimal elements to zero.
\item if $\rho(k_1)< \rho (k_2)$, then $k_1\prec k_2.$
\item if $\map(k_1)\ne 0$, then $\map(1/k_1)=1/\map(k_1)$.
\item If a sequence $(k_n)$  in  $\Kf$ converges to  $k$  in  $ \K$, then the sequence $(\rho(k_n))$ is eventually constant and equal to  $\rho(k)$ in $ \mathbb{R} $.
\end{enumerate}
\end{lemma}
\begin{proof}
	The proof is straightforward.
\end{proof}
 To  relate a  random walk to a non-Archimedean graph $ b $ we recall the normalized weight $ p(x,y)=b(x,y)/b(x) $, $ x,y\in V $. Clearly, $ p $ takes values in $ \Kf $.

\begin{definition}\label{def::rhoU}
Let     $\subsV\subseteq V$. We  let $ \pi=\pi_{U}^{b} $ be the \emph{real transition matrix for $ b $}  given by
$$
\prob(x,y)=
\begin{cases}
\map(p(x,y)), &x\in \subsV,\\
1,& x\in V\setminus \subsV\mbox{ and }x=y,\\
0,& x\in V\setminus \subsV\mbox{ and }x\ne y.\\
\end{cases}
$$
\end{definition}

Note, that if $\Bbb K=\Bbb R$, then $\pi^{b}_{U}=p_{U}$ is the classical  transition matrix for graphs with boundary (or without boundary if $\subsV=V$).

\begin{lemma}\label{lem::absStates}
	Let   $\subsV\subseteq V$. Then $\pi=\pi_{U}^{b}$ is a transition matrix whose diagonal vanishes outside of the absorbing states which are exactly  the elements of $ V\setminus\subsV $.
\end{lemma}
\begin{proof}
	For every $x\in \subsV$, we obtain $0\le \rho(p(x,y))\le 1$ by the properties of the mapping $\map$, since $0\preceq p(x,y)\preceq 1$. For $ x\in \subsV $, we have  by  linearity and local finiteness of $\map$ that
	$$
	\sum_{y\in V} \prob(x,y)=\sum_{y\in V} \map(p(x,y))=\map\left(\sum_{y\in V} p(x,y)\right)=\map(1)=1.
	$$
	Hence, $ \pi $ is a transition matrix. Moreover, for $ x\in \subsV $, we have
	$\prob(x,x)=\rho(p(x,x))=\rho(b(x,x)/b(x))=\rho(0)=0$ since $ b $ vanishes on the diagonal. Finally,  for $x\in V\setminus \subsV$, we have $\prob(x,x)=1$ by definition of $\pi$ which finishes the proof. 
\end{proof}

The lemma shows that we are indeed dealing with a transition matrix which gives rise to a random walk. 
This  allows us to give a  definition of whether a vertex in a weighted non-Archimedean graph is recurrent or transient.

\begin{definition}[Recurrent and transient vertices]
	Let $b$ be a weighted graph over $ V $,    $\subsV\subseteq V$ and $ \pi=\pi_{U}^{b} $ be the corresponding transition matrix. We call a vertex $ x\in U $ \emph{recurrent} (respectively \emph{transient}) with respect to $ U $ if $ x $ is a  recurrent (respectively transient) state with respect to $ \pi $.
\end{definition}

\subsection{Strongly connected and essential components}

We start by exploring some properties of the transition matrix $ \pi $ arising from a weighted non-Archimedean graph. Our aim is to relate directed paths with respect to $ \pi $ to specific  paths with respect to $ b $. This will lead us to our first criterion on recurrence and transience, Theorem~\ref{thm:rectranscomponent} dependent on the type of component vertices are found in.

\begin{lemma}\label{lemma::pxygec}
Let $b$ be a  graph over $V$, $\subsV\subseteq V$ and $ \prob=\prob_\subsV^b$ be the corresponding transition matrix. Then $\prob(x,y)>0$ (i.e., $x\leadsto y$) if and only if  one of the following condition holds:
\begin{enumerate}[label=$(\alph*)$]
\item
$x\in V\setminus \subsV$ and $x=y$,
\item
$x\in \subsV$ and $p(x,y)\simeq 1$.
\end{enumerate}
In particular, for $x,y\in \subsV$, we have $\prob(x,y)>0$ if and only if $p_\subsV(x,y)\simeq 1$. Moreover, if  $b(x,y)\pp b(x)$ for  $x,y\in V$, $x\ne y$, then
$
\prob(x,y)=0.
$
\end{lemma}

\begin{proof}  If (a) holds then $ \pi(x,x)=1 $ by definition. If (b) holds then  $p(x,y)\succeq c$ for some $c\in \Bbb Q^+$. Hence, $ \pi(x,y) \ge c $ since $ \rho $ is order preserving.\\
 Assume $ \pi(x,y)>0 $ and (a) does not hold. If $ x=y $ and $ x\in U $, then $ p(x,y)=0 $ and, thus, $ \pi(x,y)=0 $. If $ x\neq y $ and $ x\in V\setminus U $, then $ \pi(x,y)=0 $ by definition. Hence, $ x\neq y $ and $ x\in U $.  Finally, from $ 0<\pi(x,y)=\rho(p(x,y)) $, we infer that $ p(x,y)\simeq 1 $. This shows the equivalence.
 
The first ``in particular'' statement is  immediate and the second is also clear as $ p(x,y)=b(x,y)/b(x) $ is infinitesimal in this case.
\end{proof}

Recall that we denote $x\to y$ for two vertices if there is a directed path from $ x $ to $ y $ with respect to a transition matrix.

\begin{lemma}\label{lemma:xtoypn}
Let  a  graph $b$  over $V$, a subset $\subsV\subseteq V$ and $\pi=\pi_{U}^{b}$ be given. For two vertices $x,y\in U $, we have $x\to y$     if and only if there exist $n\in \Bbb N_0$ such that $p^{(n)}_\subsV(x,y)\simeq~1$.
\end{lemma}

\begin{proof}
If $x=y$, then $p^{(0)}_\subsV(x,x)=1$ and $x\to x$ by definition, so there is nothing to prove. 

So, let $x\ne y$. Note that $p^{(n)}_\subsV(x,y)\simeq 1$ is equivalent to the existence of $c\in \Bbb Q^+$ such that $p^{(n)}_\subsV(x,y)\succ~c$.

Let $x\to y$. Then there exists a directed path between $x$ and $y$
\begin{equation*}
x=x_0\leadsto x_1\leadsto x_2\leadsto \dots \leadsto x_n=y,
\end{equation*}
and we can assume that $x_i\ne x_j$ for $ i\neq j $ on the path. Then by (b) of Lemma \ref{lemma::pxygec} and since $y\in \subsV$ we obtain that $x_0,\ldots, x_{n}\in \subsV$ and that there are $c_0,\ldots,c_{n-1}\in \Bbb Q^+$ such that $p_\subsV(x_{i}, x_{i+1})=p(x_{i}, x_{i+1})\succ c_i$ for $ i=0,\ldots,n-1 $. Therefore, we obtain $p^{(n)}_\subsV(x,y)\simeq~1$ since
$$
p_\subsV^{(n)}(x,y)\succeq p(x,x_1)p(x_1,x_2)\dots p(x_{n-1},y)\succ c_0\cdot c_1\dots c_{n-1} =:c.
$$

Assume now $p^{(n)}_\subsV(x,y)\simeq~1$, i.e., $p_\subsV^{(n)}(x,y)\succ c\in \Bbb Q^+$.  Assume, that for a fixed $ n $ all  paths  $x
=x_{0}\sim \ldots\sim x_{n}=y$ in $\subsV$ with respect to $ b $
one has $$ p_\subsV(x,x_1)p_\subsV(x_1,x_2)\dots p_\subsV(x_{n-1},y)\pp 1 .$$
Then, due to the local finiteness  there exists an infinitesimal $\tau\in \K^+$ such that for any such path we have 
$$
 p_\subsV(x,x_1)p_\subsV(x_1,x_2)\dots p_\subsV(x_{n-1},y) \preceq \tau
$$
and consequently
$$
p_\subsV^{(n)}(x,y)\preceq \#\{\mbox{undirected paths from $x$ to $y$}\}\tau
$$
is infinitesimal, which is a contradiction. Therefore, there exists an undirected path $x_0\sim\ldots\sim x_{n-1}$ in $ \subsV $ from $x$ to $y$ 
such that
$$
p_\subsV(x,x_1)p_\subsV(x_1,x_2)\dots p_\subsV(x_{n-1},y)\succ c'\in \Bbb Q^+.
$$
Now, since all the factors are smaller than $1$, any of those factors can be bounded from below with a positive rational number. Hence,
$ x=x_0\leadsto x_1\leadsto  x_2\leadsto \dots\leadsto  x_n=y $
 which gives $x\to y$.
\end{proof}
The next theorem relates paths for $ \pi $ to  paths for $ b $ with certain additional properties.

\begin{theorem}\label{thm::bxby}
Let $b$ be a  graph over $V$, $\subsV\subseteq V$, $x,y\in \subsV$ and $\pi=\pi_\subsV^{b}$. Then the following holds:
\begin{enumerate}[label=${(\alph*)}$]
\item We have 
$x\to y$ if and only if there exists a path $$ x=x_0\sim x_1\sim\dots \sim x_n=y $$ in  $\subsV$ (with respect to $ b $) such that for all $ i=0,\ldots,n-1 $
 $$ b(x_i,x_{i+1})\simeq b(x_i). $$
In particular, in this case 
${b(x_{i+1})\succsim b(x_i)}$, $ i=0,\ldots,n-1$
and ${b(y)}\succsim{b(x)}$.
Furthermore, in this case this path is also a directed path (with respect to $ \pi $)
$$
x=x_0\leadsto x_1\leadsto\dots \leadsto x_n=y
$$
and any such directed path is a path with respect to $ b $ with the properties above. 
\item
Let $x\to y$. Then $y\to x$ if and only if  $b(x)\simeq b(y).$

\item
 We have  $x\leftrightarrow y$ if and only if $b(x)\simeq b(y)$ and there exists a path  $$ x=x_0\sim x_1\sim\dots \sim x_n=y $$ in  $\subsV$  such that for all $ i=0,\ldots,n-1 $
with $$ b(x_i,x_{i+1})\simeq b(x). $$ 
In particular, this path is also a directed path form  $ x $ to $ y $ and from $ y $ to $ x $.
\end{enumerate}
\end{theorem}

\begin{proof}
Observe that if $ v,w\in U $, then $ p_{U}(v,w)=p(v,w) $ which we will apply throughout the proof.

(a)
Assume $x\to y$. By Lemma \ref{lemma:xtoypn} $p_\subsV^{(n)}(x,y)\simeq 1$ for some $n\in \Bbb N$, i.e,. there exist $n\in ~\Bbb N, c\in~\Bbb Q^+$ such that $p_\subsV^{(n)}(x,y)\succ c$. Then there exists a path 
$
x=x_0\sim x_1\sim\dots \sim x_n=y
$
with $p_\subsV(x,x_1)\dots p_\subsV(x_{n-1},y)\succ c'$ for some $c'\in \Bbb Q^+$ (if all the paths would be infinitesimal, then $p_\subsV^{(n)}$ would be infinitesimal). 
Since $p_\subsV\preceq 1$, we have 
$$
1\succeq\dfrac{b(x_i,x_{i+1})}{b(x_i)}=p(x_i,x_{i+1})=p_\subsV(x_i,x_{i+1})\succ c'
$$
and we immediately get $b(x_i,x_{i+1})\simeq b(x_i)$ for  $ i=0,\ldots,n-1$. 

For the other direction, assume there exists a  path $
(x_1,\ldots,x_{n})$ as described. By finiteness of the  path,  there exist $c\in \Bbb Q^+$ such that
$$
1\succeq p(x_i,x_{i+1})=p_\subsV(x_i,x_{i+1})=\dfrac{b(x_i,x_{i+1})}{b(x_i)}\succ c.
$$
Thus, we have $x_i\leadsto x_{i+1}$ due to Lemma \ref{lemma::pxygec} and, hence, $x\to y$.

The ``in particular'' statement follows from the fact, that $b(x_{i+1})\succeq b(x_{i},x_{i+1})$ for all $ i =0,\ldots,n-1$ and transitivity of $ \succeq $. The final statement of (a) was shown along the proof.\medskip

(b) 
Let $x\to y$. 
Then,  by Lemma  \ref{lemma:xtoypn}, $ x\to y $ implies the existence of $c\in \Bbb Q^+$ such that $p_\subsV^{(n)}(x,y)\succ c$.
Since, 
$$
p_\subsV^{(n)}(y, x)=\dfrac{b(y)}{b(x)}p_\subsV^{(n)}(x,y)
$$
a lower bound by a positive rational on $ p_\subsV^{(n)}(y, x) $ is equivalent to a corresponding lower bound on ${b(y)}/{b(x)} $ (as $ p_{U}\succeq 1 $).  On the other hand, the ``in particular'' statement also yields an upper bound by a  rational on ${b(y)}/{b(x)} $ in the case $ y\to x $.\medskip

(c) This follows directly from $(a)$ and $(b)$.
\end{proof}

{\begin{corollary}\label{cor310}
If $x\leftrightarrow y$ and $(x_0,x_{1},\ldots,x_{n})$ is a directed path from $x$ to $y$, then $(x_n,x_{n-1},\ldots,x_{0})$ is a directed path from $y$ to $x$, i.e.,
$
y=x_n\sim x_{n-1}\sim\dots \sim x_0=x.
$
\end{corollary}
\begin{proof}
This follows from proof of $(a)$ of Theorem \ref{thm::bxby}, since  $b(x)\simeq b(y)$ by $(b)$.
\end{proof}}

These relations between a graph $ b $ and its transition matrix $ \pi $ give rise to the following definition.

\begin{definition}[Strongly connected components]\label{def::StronglyCCGr}
Let $b$ be a  graph over $V$ and $\subsV\subseteq V$.  A \emph{strongly connected component with respect to $U$} is a maximal subset $C\subseteq U$  (with respect to  inclusion) such that for any $x,y\in C$ 
there exists a path $
x=x_0\sim x_1\sim\dots \sim x_n=y
$ in $ U $ such that
 for all $i={0,\ldots, n-1}$ $$ b(x)\simeq b(x_i,x_{i+1})\simeq b(y). $$
\end{definition}
 The following corollary is an immediate consequence of Theorem \ref{thm::bxby}.
\begin{corollary}\label{Corollary:classesb(x)}Let $b$ be a  graph over $V$ and $\subsV\subseteq V$. 
The strongly connected components of $ \pi_{U}^{b}$, which are not absorbing states, are exactly the strongly connected components of $b$ restricted to  $\subsV$.
\end{corollary}

Furthermore, we can now define the notion of an essential or absorbing component.
\begin{definition}[Essential components]\label{def::StronglyCCGr}
	Let $b$ be a  graph over $V$ and $\subsV\subseteq V$.  A {strongly connected component $ C $ with respect to $U$} is called \emph{essential} (or \emph{absorbing}) if for every $ x\in C $ and every path $
	x=x_0\sim x_1\sim\dots \sim x_n=y
	$  in $ U $ such that
	for all $i={0,\ldots, n-1}$  we have $$  b(x_i,x_{i+1})\simeq b(x_{i}), $$
	we also have
	$$ b(x)\simeq b(y). $$
\end{definition}

The next corollary shows that essential components are exactly  the essential classes of the corresponding random walk. 

\begin{corollary}\label{Corollary:essentialclasses}Let $b$ be a  graph over $V$ and $\subsV\subseteq V$. 
	The essential components of $ \pi_{U}^{b}$, which are not absorbing states, are exactly the strongly connected components of $b$ restricted to  $\subsV$.
\end{corollary}
\begin{proof}
This an immediate consequence of Theorem~\ref{thm::bxby} and the definition of an essential class.
\end{proof}

This corollary allows  for a first necessary condition on recurrence.
\begin{theorem}[A first recurrence/transience criterion]\label{thm:rectranscomponent} Let $b$ be a  graph over $V$, $\subsV\subseteq V$ and $ \pi =  \pi_{U}^{b} $.
	\begin{itemize}
		\item [(a)] If a vertex is recurrent, then it is in an essential component. Furthermore, if an essential component is finite, then every contained vertex is recurrent.
		\item [(b)]  If a vertex is not in an essential component, then it is transient.
	\end{itemize}
\end{theorem}
\begin{proof}
The result follows directly from Corollary~\ref{Corollary:essentialclasses} and Theorem~\ref{lem:recvsess}.
\end{proof}

\subsection{Characterizing the random walks obtained by the construction}

In the previous section we constructed a transition matrix $\pi= \pi_{U}^{b} $ of a random walk  for a graph $ b $ over $ V $ and a subset $ \subsV\subseteq V $. These random walk are directed weighted graphs and we next turn to the question which kind of directed graphs are obtained in terms of combinatorics.

We start with  a result which is a consequence of Lemma~\ref{lemma:xtoypn} and Theorem~\ref{thm::bxby}. It shows that distance minimizing paths with respect to $ \pi $ and the specific paths identified above with respect to $ b $ coincide.

\begin{theorem}\label{thm::directedpath} Let $b$ be a  graph over $V$ and $\subsV\subseteq V$.  Let  $x,y\in \subsV$ such that $x\leftrightarrow y$. Then the shortest directed paths from $x$ to $y$ and from $y$ to $x$ with respect to  $\pi_{\subsV}^{b}$ are of the same length. 
\end{theorem}
\begin{proof}
	It follows immediately, since by Corollary \ref{cor310} and Theorem \ref{thm::bxby} (c),  all the directed paths from $x$ to $y$ can be reverted to a directed path from $ y  $ to $ x$.
\end{proof}

This theorem allows us to give an example to exclude certain phenomena.

\begin{example}
In a graph $ b $ over a vertex set $ V$	with $\#V>2$ 
the directed graph $\pi_{U}^{b}$, $U\subseteq V$,  never includes a directed cycle of vertices $ x_1\leadsto x_2\leadsto \dots \leadsto x_n\leadsto x_{1} $ but  $ x_{i+1}\not\leadsto x_{i} $ for some $ i=1\ldots, n-1 $, e.g. see Figure \ref{fig::C4} below.
\begin{figure}[H]
\centering
\begin{tikzpicture}[auto,node distance=2cm,
                    thick,main node/.style={circle, draw, fill=black!100,
                        inner sep=0pt, minimum width=3pt}]

  \node[main node] (1) [label={[above]$x_1$}]{};
  \node[main node] (2) [above right of=1,label={[above]$x_2$}] {};
  \node[main node] (3) [below right of=2,label={[above right]$x_3$}] {};
  \node[main node] (4) [ below left of=3,label={[below=2mm]$x_4$}] {};

  \path[every node/.style={font=\sffamily\small}]
    (1) edge[color=black,  ->]  (2)
   (2) edge[color=black,  ->]  (3)
    (3) edge[color=black, ->]   (4)
    (4) edge[color=black, ->]   (1);

\end{tikzpicture}
\caption{Directed cycle  with $4$ vertices which cannot arise from a non-Archimedean graph.}
\label{fig::C4}
\end{figure}
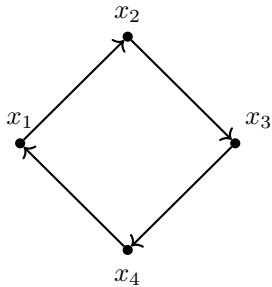
Indeed, by Corollary \ref{cor310}, since for any $i,j$ paths from $x_i$ to $x_j$ cannot be reversed, so, we obtain that $\#U=1$. Assume w.l.o.g that $U=\{x_1\}$. Then by definition of $\pi_U$ we get $\pi_U(x_n,x_1)=0$, and, hence, $x_n\not \leadsto x_1$, which is a contradiction.
\end{example}

The next lemma makes the observation of the example even more precise.

\begin{lemma}\label{Cor:neighbours} Let $b$ be a  graph over $V$,  $\subsV\subseteq V$ and $x,y\in V$, $ x\neq y $. If  $x\leadsto y$ (respectively $ x\to y $),  then exactly one of the following holds:
\begin{itemize}
\item
$y\leadsto x$ (respectively $ y\to x $), which implies $ y\in U $.
\item
$y\in V\setminus \subsV$, i.e., $y$ is an absorbing state.
\item
$y\in \subsV$, $b(x)\pp b(y)$ and they  belong to different strongly connected components.

\end{itemize}
\end{lemma}

\begin{proof}
First observe, that  $x\leadsto y$ implies that $x$ is not an absorbing state, and $x\in \subsV$ by Lemma~\ref{lem::absStates}. 
Reversing the roles of $ x $ and $ y $, says that we cannot have both $ y\leadsto x $ and $ y\in V\setminus \subsV $.

If $y\not\leadsto x$ and $y$ is not an absorbing state, i.e., $y\in U$, then $p(x,y)>0$ leads to $b(x,y)\simeq b(x)$, while $p(y,x)=0$ leads to $b(x,y)\pp b(x)$, from where the inequality $b(x)\pp b(y)$ and the statement about strongly connected component follow. 

Finally, if   $b(x)\pp b(y)$, then  $y\not\leadsto x$ by Theorem~\ref{thm::bxby}.

The corresponding statement for $ x\to y $ follows by repeating the above argument inductively along some path from $ x $ to $ y $ given by Theorem~\ref{thm::bxby}.
\end{proof}

Next, we state the criterion for a directed graph to arise from a non-Archimedean graph. To this end, we to need some basic facts on theory  of partially ordered sets.
Note that irreducible classes  of any directed graph together with relation $\to$ form a partially ordered set. We treat the  finite and infinite case separately.
\begin{definition}
For a partially ordered set $(A, \to)$  and two elements $a,b\in A$  a \emph{chain of length $ n $ from $a$ to $b$}  is any finite sequence $(a_1,\ldots,a_{n})$  in  $ A$ such that 
$$
a=a_0\to a_1\to \dots\to a_n=b.
$$
The \emph{height} $ h(a) $  of $ a\in A $ is the supremum of the lengths of paths from any element of $ A $ to $ a $.
\end{definition}
Observe that if $ A  $ is finite then $ h(a) $ is the maximal length of chains from minimal elements to $ a $ and, in particular, finite. In this case, it is immediate to see that $ h $ is order preserving, i.e., if $ a\to b $, then $ h(a)<h(b) $.

\begin{theorem}\label{thm:finitevsdirected} Let $ \pi $ be a transition matrix over a finite set $ V $. Then, there exists a weighted graph $ b $ over $ V $ with respect to a non-Archimedean ordered field $\Bbb K$ and $\subsV\subseteq V$ such that $\pi=\pi_{U}^{b}$ if and only if  there is a function $\rw:V\to \Bbb R^+$ such that 
	for any irreducible class $C$ and  $x,y\in C$ 
	$$
	\prob(x,y)\rw(x)=\prob(y,x)\rw(y)
	$$
	and $ \pi(x,x)=0 $ if $ x $ is not an absorbing state.
	Furthermore, for a given transition matrix, $ \beta $ can be chosen such that $ \beta (a)=1 $ for some $ a\in C $ and $  \beta(x)=\rho (b(x)/b(a)) $ for all $ x\in C $.
\end{theorem}

\begin{proof}
	We  construct $b$ and $U$. First, we fix an infinitely large $\N \in\K$, i.e., $1\pp\N$. Furthermore, we know that the set of irreducible classes is a finite partially ordered set and has a height function $ h $.  For  $ x \in  V$ contained in a irreducible class  $ C_{x} $, we denote $ h(x)=h(C) $.
	We let $ U $
	be the subset of $ V $ which are not absorbing states. To define $ b $,  we let $ b=0 $ on $ (V\setminus \subsV) \times(V\setminus \subsV) $. Furthermore, we define
	for  $x\in U$,  $y\in C_{x}$
	\begin{align*}
		b(x,y) 
		=	\prob(x,y)\rw(x) \N^{h({x})}
	\end{align*}
and  for $ y\in V\setminus C_{x} $
	\begin{align*}
b(x,y) 
=		\max\{\prob(x,y)\rw(x),\prob(y,x)\rw(y)\}\N^{\min\{h({x}), h({y})\}}.
\end{align*}
Then $ b $ is clearly symmetric and has zero diagonal.
Observe,  for $ x\in U $,
\begin{align*}
	b(x,y)=
	\begin{cases}
		\prob(x,y)\rw(x) \N^{h({x})} &: x\leadsto y, \mbox{ i.e., $ \prob(x,y)>0 $}\\
		\prob(y,x)\rw(y) \N^{h({y})} &: \mbox{else}.
	\end{cases}
\end{align*}
Note that in the second case we have $ b(x,y)=0 $ whenever $ y\not\leadsto x $ and $ h(y)<h(x) $ whenever $ y\leadsto x $.
As $ \prob(x,x)=0 $ for $ x\in U $, we have 
\begin{align*}
	  b(x)  &=\sum_{y:x\leadsto y}b(x,y)+\sum_{\substack{y:x\not\leadsto y,\\ \;\; y\leadsto x }}b(x,y)=\sum_{ y\in V } 	\prob(x,y)\rw(x) \N^{h({x})} +
\sum_{\substack{y:x\not\leadsto y,\\ \;\; y\leadsto x}} 	\prob(y,x)\rw(y) \N^{h({y})}\\
& =  \left(  1+ \sum_{\substack{y:x\not\leadsto y,\\ \;\; y\leadsto x} } 	\prob(y,x)\frac{\rw(y)}{\rw(x)} \N^{- (h(x)-h({y}))}\right)\rw(x)\N^{h(x)}=(1+\tau)\rw(x)\N^{h(x)}
\end{align*}
for some infinitesimal $ \tau $. We obtain for $ \prob(x,y)>0 $ and $ x\neq y $ (which implies $ x\in U $)
\begin{align*}
	\pi^{b}_{\subsV}(x,y) = \rho(p(x,y))=\rho\left(\frac{b(x,y)}{b(x)}\right)
	=\rho\left(\frac{\prob(x,y)\beta(x)\N^{h(x)}}{(1+\tau)\beta(x)\N^{h(x)}}\right)
	=\rho\left(\frac{\prob(x,y)}{(1+\tau)}\right)=\prob(x,y)
\end{align*}
and for $ \prob(x,y)=0 $ (which includes the case $ x\in V\setminus U $ and $ x\neq y $) where  we have either $ \prob(y,x)=0 $ or $ h(y)<h(x) $
\begin{align*}
	\pi^{b}_{\subsV}(x,y) =\rho\left(\frac{\prob(y,x)\beta(y)\N^{h(y)}}{(1+\tau)\beta(x)\N^{h(x)}}\right)= \rho\left(N^{-(h(x)-h(y))}\frac{\prob(y,x)\beta(y)}{(1+\tau)\beta(x)}\right)=0=\prob(x,y).
\end{align*}
Finally, $ \pi_{U}^{b}(x,x) =1$ by definition for $ x\in V\setminus \subsV $.\medskip

To prove the opposite direction we fix in each  irreducible class $C$ a vertex $x\in C$ and put
 $\rw(x)=1.
$
Further, for any $y\in C$, we prescribe $\rw(y)=\map({b(y)}/{b(x)})$. Then, we have, for any $y,z\in~C$,
\begin{align*}
\prob(y,z) \rw(y)=\map(p(y,z))\map\left(\dfrac{b(y)}{b(x)}\right)=\map\left(\dfrac{p(y,z) b(y)}{b(x)}\right)=\map\left(\dfrac{p(z,y) b(z)}{b(x)}\right)=\prob(z,y) \rw(z),
\end{align*}
where the second and last equality hold, since $p\preceq 1$ and ${b(y)}/{b(x)}$ as well as ${b(z)}/{b(x)}$  have no infinitely large component, since $y,z$  and $ x$ are all in the same irreducible class  $C$, confer Theorem~\ref{thm::bxby}.
\end{proof}

The idea of proof for infinite graph is  similar, but there is a difference: we can not necessarily map an infinite partially ordered set to natural numbers, see e.g. \cite{mathoverflow}.
 To this end, we prove  the following lemma.

\begin{lemma}\label{lemma:mappingDinf}
Let $(A, \to)$ be a countable partially ordered set. Then, there exists an order preserving  mapping
$
h_\infty:A\to \Bbb Q.
$\end{lemma}

\begin{proof}
Since $ A $ is countable, we can enumerate it (in an arbitrary not necessarily order preserving way), i.e., $ A=\{a_{0},a_{1},a_{2},\ldots\} $. 
Let $ A_{n}=\{a_{0},\ldots,a_{n}\} $ and we  define $ h_{\infty}(a_{n}) $ inductively over $ n $. Let 	$ h_{\infty}(a_{0})=0 $. Having defined $ h_{\infty}(a_{0}) ,\ldots, h_{\infty}(a_{n})$ in an order preserving way. Then, there exists an $ q\in \mathbb{Q} $ such that
$$
\max_{b\in A_{n}, b\to a_{n+1}} h_\infty(b)<q< \min_{b\in A_{n}, a_{n+1}\to b} h_\infty(b)
$$
and we put $  h_\infty(a_{n+1})=q $.
\end{proof}

We now come to an analogue of Theorem~\ref{thm:finitevsdirected} for infinite graphs.

\begin{theorem}\label{thm:infinitevsdirected}  Let $ \pi $ be a locally finite  transition matrix over an at most countable set $ V $. Then, there exists a weighted graph $ b $ over $ V $ with respect to a real closed non-Archimedean ordered field and $\subsV\subseteq V$ such that $\pi=\pi_{U}^{b}$ if and only if  there is a function $\rw:V\to \Bbb R^+$ such that 
	for any irreducible class $C$ and  $x,y\in C$ 
	$$
	\prob(x,y)\rw(x)=\prob(y,x)\rw(y)
	$$
	and $ \pi(x,x)=0 $ if $ x $ is not an absorbing state. 	Furthermore, $ \beta $ can be chosen such that $ \beta (a)=1 $ for some $ a\in C $ and $  \beta(x)=\rho (b(x)/b(a)) $ for all $ x\in C $.
\end{theorem}

\begin{proof}
 Lemma \ref{lemma:mappingDinf} allows us to reproduce the proof of Theorem~\ref{thm:finitevsdirected} by replacing $ h $ with $h_\infty$. Note, that we need a real closedness in order to have rational  powers of an infinitesimal $\mathcal N$.
\end{proof}

\begin{remark}
	Given the height function $ h $, 
the proof of Theorem \ref{thm:finitevsdirected} is an explicit construction and also $ h $ can be obtained explicitly from a finite random walk. However, in the infinite case, the inductive construction of $ h_{\infty}  $ often does not allow for an explicit construction. In consequence, the infinite graphs arising from 
Theorem~\ref{thm:infinitevsdirected} are not as explicit as the crucial ingredient $ h_{\infty} $ is not.
\end{remark}

\section{The  reciprocal of the normalized capacity}\label{sec:G}

In this section we study a quantity $ G  $ which corresponds to the diagonal of the Green's function.  In order to define this quantity we need the following lemma.

\begin{lemma}\label{cor::GassolOfDP} Let $b$ be a  graph over $V$. For  finite connected sets $K\subseteq L\subseteq V$ and  $a\in K$,
	\begin{equation*}
	\dfrac{\c_L(a)}{b(a)}\preceq	\dfrac{\c_K(a)}{b(a)}\preceq 1.
	\end{equation*}
Furthermore,
	\begin{equation*}
	\dfrac{1}{\displaystyle\rho\left(\frac{\c_L(a)}{b(a)}\right)}\ge 	\dfrac{1}{\displaystyle\rho\left(\frac{\c_K(a)}{b(a)}\right)}=\dfrac1{\sum_{x\in V}{(1-\map(v(x)))\map(p(a,x))}}\ge 1.
	\end{equation*}
	where $v$ is the solution of Dirichlet problem \eqref{dirpr} for $ K $ and $ a $.
\end{lemma}
\begin{proof}
By monotonicity of the capacity, Lemma~\ref{lem::monOfcap}, and the  maximum principle, Proposition \ref{thm::maxpr}, we get
	$$
	\dfrac{\c_L(a)}{b(a)}\preceq \dfrac{\c_K(a)}{b(a)}=\Delta v(a)=\dfrac{1}{b(a)}\sum_{x\in V} {(1-v(x))}b(a,x)\preceq \dfrac{\sum_{x\in V} b(a,x)}{b(a)}\preceq 1,
	$$
	where $v$ is the solution of Dirichlet problem \eqref{dirpr}  for $ K $ and $ a $ which satisfies $ v\succ 0 $. Thus, we can apply $ \map $ to the left hand side and the second statement follows as $ \map $ is an order preserving ring morphism and the definition of $p(a,x)=b(a,x)/b(a)  $.
\end{proof}

This allows us to define the quantity which is vital for the further considerations of this paper.

\begin{definition}
Let $b$ be a  graph over $V$. For finite connected $K\subseteq V$ and $a\in K$, we define
$$
\G_K(a)=\dfrac{1}{\displaystyle\rho\left(\frac{\c_K(a)}{b(a)}\right)}.
$$
Furthermore, for an exhaustion $ (K_{n}) $ of $ V $ and $ a\in K_{0} $, we define
$$
\G(a)=\limsup_{n\to\infty}\G_{K_n}(a)\in [1,\infty].
$$
\end{definition}
Observe that due to Lemma~\ref{cor::GassolOfDP}, we have $$  1\leq\G_{K_{n}}(a)\leq \G_{K_{n+1}}(a)\leq \G(a) $$ for $ n\in \mathbb{N} $. Thus, the limit superior above is indeed a limit. Furthermore, if the graph has positive or null capacity, i.e., $ \c(a) $ exists, then
$$
\G(a)=\dfrac{1}{\map\left(\dfrac{\c(a)}{b(a)}\right)},
$$
where $ 1/0 = \infty$.  Next, we discuss a local Harnack inequality which will eventually allow us to conclude that $ \G $ is either finite or infinite on strongly connected components.

\begin{lemma}[Local Harnack inequality]\label{lemma::Harnack}
Let $b$ be a  graph over $V$ and let $v$ be the solution of Dirichlet problem \eqref{dirpr} for a finite connected $K\subseteq V$ and $ a\in K $. If $x,y\in V$ such that $x\to y$ with respect to $ \pi_{K}^{b} $, then
 $v(x)\succsim v(y)$.  In particular, if $x$  and $y$ belong to the same irreducible class, then $v(x)\simeq v(y)$. Furthermore, the constants in the two sided  estimates are uniform when increasing $ K $. 
\end{lemma}

\begin{proof}
If $x=a$ or if $y\not\in K$, then the statement is obvious since $0\preceq v\preceq 1=v(a)$ by maximum principle, Proposition~\ref{thm::maxpr}. So assume $x\in K$, since otherwise  $x$ is an absorbing state and there is no  $y\ne x$ with $x\to y$.
By Theorem \ref{thm::bxby} there exists a path $(x_1,\ldots,x_{n})$  in  $ K$ and $b(x_i)\simeq b(x_i,x_{i+1})$, i.e., ${b(x_i,x_{i+1})}/{b(x_i)}\succeq c$ for some $c\in \Bbb Q^+$. 

If $x_i\ne a$ for all $i$, we have by $\Delta v(x_i)=0$ that
$$
v(x_i)=\dfrac{\sum_{y\in V} b(x_i,y)v(y)}{b(x_i)}\succeq \frac{b(x_i,x_{i+1}) }{b(x_i)}v(x_{i+1})\succeq c v(x_{i+1}),
$$
since $v\succeq 0$ by the maximum principle. We conclude  $v(x)\succsim v(y)$ by induction. Further observe that the path above is included in any finite set including $ K $.
 
If there exists $i_0$ such that $x_{i_0}=a$, then we can conclude from  $\Delta v(x_i)=0, i< i_0$, in the same way as above, that $v(x)\succeq c v(a)=c$. Since $v(y)\preceq 1$ by maximum principle, Theorem \ref{thm::maxpr}, the statement follows.
\end{proof}

Recall that a strongly connected component with respect to $ U $ is a  maximal subset $W\subseteq U$  such that for any $x,y\in W$ 
there exists a path $
x=x_0\sim\dots \sim x_n=y
$ in $ U $ such that
$ b(x)\simeq b(x_i,x_{i+1})\simeq b(y) $ for all $ i $. These are exactly the strongly connected components of $ \prob_{U}^{b} $ which are not absorbing states, Corollary~\ref{Corollary:classesb(x)}.

\begin{theorem}\label{thm:Harnack}
Let $b$ be a  graph over $V$ and $ x,y\in K $.
\begin{itemize}
	\item [(a)] If $ K\subseteq  V $ is finite and connected and if $x$ and $y$ belong to the same strongly connected component with respect to $K$, then $\G_K(x)$ and $\G_K(y)$ are either both finite or both infinite. 
	\item [(b)] If $x$ and $y$ belong to the same strongly connected component with respect to $V$, then $\G(x)$ and $\G(y)$ are either both finite or both infinite. 
\end{itemize}
\end{theorem}
\begin{proof}
By Lemma \ref{lem::FiniteSolxy} we have
$$ 
\dfrac{v^x(y)}{\c_K(x)}=\dfrac{v^y(x)}{\c_K(y)}.
$$
and by Lemma \ref{lemma::Harnack} $v^x(y)\simeq v^x(x)=1=v^y(y) \simeq v^y(x)$.  Hence $\c_K(x)\simeq \c_K(y)$. Since $b(x)\simeq b(y)$ by Theorem \ref{thm::bxby}, 
we have 
\begin{align*}
	\frac{\c_K(x)}{b(x)}\simeq \frac{\c_K(y)}{b(y)}
\end{align*}
and the first statement follows due to the definition of $\G_K$. For the second statement we observe that for an exhaustion $(K_{n})  $ the capacity $( \c_{K_{n}}) $ is monotone decreasing and the constants in the two sided estimate are uniform in the increasing sequence $ (K_{n}) $ by the Harnack inequality, Lemma~\ref{lemma::Harnack}. Thus, we infer the statement also for the limit.
\end{proof}

We end this section with a lemma that shows that the real part of the solution of the Dirichlet problem  for a vertex $ a $ is constant on all irreducible classes on which $ b $ is infinitely larger than $ b(a)$. Recall that the solution $ v $ of the  Dirichlet problem \eqref{dirpr} satisfies $ 0\preceq v\preceq  1$ by the maximum principle, Proposition~\ref{thm::maxpr} and we can therefore study $ \rho(v) $.

\begin{lemma}
Let $b$ be a  graph over $V$ and  $v$ be the solution of Dirichlet problem \eqref{dirpr} for a finite connected set $K$ and $a\in K$. Let $ x \in V$ be such that $b(x)\ss b(a)$. Then   $\map(v)$ is constant on the irreducible class which contains $ x $.
\end{lemma}

\begin{proof} Let $ C $ be the irreducible class which contains $ x $. Then $ b(y)\simeq b(x) $ for all $ y\in C $ by Theorem~\ref{thm::bxby}.
Let $x_1, x_2\in C$ be such that $\prob(x_1,x_2)>0$. Then $b(x_1,x_2)\simeq b(x_1)\ss b(a)$. Assume that $\map(v(x_1))\ne \map(v(x_2))$, i.e., $|v(x_1)-v(x_2)|\simeq 1$. Then, since $ \c_{K}(a)=Q (v) $, we obtain
$$
\dfrac{\c_K(a)}{b(a)}=\dfrac{\sum_{x,y\in V} b(x_1,x_2)(v(x_1)-v(x_2))^2}{b(a)}\succsim \dfrac{b(x_1,x_2)}{b(a)}\ss 1,
$$
which is a contradiction to Lemma~\ref{cor::GassolOfDP}.
\end{proof}

\section{A  characterization for recurrence and transience}\label{sec:char}

In this section we study recurrence and transience in view of $ \G $. First of all notice that a first determination of whether a vertex is recurrent or transient can be made on whether it is contained in an essential component, Theorem~\ref{thm:rectranscomponent}. Hence, the open question remains  to determine the type of a state in essential components which will be completely characterized in terms of $ \G $ in Theorem~\ref{thm:charGesscomp}. However, we will indeed prove more namely that the implication that finiteness of $ \G(a) $ tells us that $ a $ is transient, Corollary~\ref{cor::rec}. One may wonder whether also the reverse direction holds outside of essential components. This is not the case and will be shown by counterexamples at the end of the section.

Before we start, we recall what is classically known for graphs $ b $ over the reals.  Then \cite[3.4 Theorem]{Woess09} tells us that 
\begin{align*}
		\dfrac{b(a)}{\c_K(a)}=\sum_{n=0}^{\infty}P_{K}^{n}1_{a}(a)\qquad\mbox{and}\qquad 	\dfrac{b(a)}{\c(a)}=\sum_{n=0}^{\infty}P^{n}1_{a}(a)
\end{align*}
for any finite subset $ K\subseteq V $ and $ a\in K $. In the second formula, the both sides are infinity if $ \c(a)=0 $ which is the case if $ a $ is a recurrent state. 

For the non-Archimedean case, these  equalities do only hold  if right hand sides converge, see \cite[Theorems 8.2 and 8.5]{OurPreprint}.

In this section we prove a more general statement, using $\G_K(a)$  (resp. $\G(a)$) instead of capacity and the real transition matrix of the graph instead of the non-Archimedean one.

\begin{theorem}
	Let $b$ be a  graph over $V$, let $K\subseteq V$ be finite and connected, $ \prob_{K}=\prob_{K}^{b} $ and $ \prob=\prob_{V}^{b} $. Then, for any $a\in K$ 
\begin{align*}
	\G_K(a)\ge \displaystyle\sum_{n=0}^\infty\prob^{(n)}_K(a,a)\qquad \mbox{and}\qquad
	G(a)\ge \sum_{n=0}^\infty\prob^{(n)}(a,a),
\end{align*}
	where infinity is allowed on both sides in the second inequality.
\end{theorem}
\begin{proof}
We calculate using that $ \Delta_{K}=I_{K}-P_{K} $ is invertible and a telescopic argument
	\begin{multline*}
		\Delta^{-1}_{K} 1_a(a)-\sum_{n=0}^{N}P_{K}^n 1_a(a)=\left(\Delta^{-1}_{K}-\sum_{n=0}^{N}P_{K}^n\right) 1_a(a)\\
		=\Delta^{-1}_{K}\left(I_{K}-(I_{K}-P_{K})\sum_{n=0}^{N}P_{K}^n\right) 1_a(a)=\left(\Delta^{-1}_{K}P_{K}^{N+1}\right) 1_a(a)\succeq 0,
	\end{multline*}
	for all $ N $ due to the non-negativity in $\mathbb{K}$ of both matrices, Lemma~\ref{lem:RenDP} and the definition of $P_K$. 
	Hence, as a consequence of Lemma~\ref{lem:RenDP} and the definition of the capacity, we get
	$$
\frac{1}{\frac{\c_K(a)}{b(a)}}=	\dfrac{b(a)}{\c_K(a)}=
	\Delta^{-1}_{K} 1_a(a)
	\succeq
	\sum_{n=0}^{N}P_K^n  1_a(a).
	$$
We take $\rho$ on both sides and then take the limit $ N\to\infty $. This gives the first inequality. For the second inequality, we let $ K $ be equal to an exhaustion $ (K_{n}) $ and take the monotone limit.
\end{proof}

This immediately gives rise to the following corollary.

\begin{corollary}\label{cor::rec} Let $ a\in V $.
	\begin{itemize}
		\item [(a)] If $ a $ is a recurrent state, then $\G(a)=\infty$.
		\item [(b)] If $ \G(a)<\infty $, then $ a $ is a transient state.
	\end{itemize} 
\end{corollary}

Let a graph $ b $ and $ \pi=\pi_{\subsV}^{b} $ for $ \subsV      \subseteq     V$. Then, Theorem~\ref{thm:infinitevsdirected}  gives the existence of $ \beta: V\to \mathbb{R}^{+} $ such that $
\prob(x,y)\rw(x)=\prob(y,x)\rw(y),
$
for $ a\in C\subseteq V $ in an irreducible class $C$, where  $ \beta $ can be chosen such that for some fixed $ a\in C $ and all other $ x\in C $ $$  \beta (a)=1 \qquad\mbox{and}\qquad  \beta(x)=\rho \left(\frac{b(x)}{b(a)}\right) . $$  This allows us to define 
\begin{align*}
	\beta(x,y)=\prob(x,y)\rw(x) 
\end{align*}
for all $ x,y\in V $. Then, $ \beta $
 is symmetric on irreducible classes and as $ \pi $ is a transition matrix which sums up to $ 1 $ about each vertex
 \begin{align*}
 	\beta(x)=\sum_{y\in V}\pi(x,y)\beta(x) =\sum_{y\in V}\beta(x,y).
 \end{align*}
Observe that for two irreducible classes $ C $ and $ C' $ and $ x\in C $ and $ y\in C' $ such that $ x\leadsto y $ we have $ \beta(x,y)>0 $ but $ \beta(y,x) =0$. Hence, $ \beta $ is not necessarily symmetric on $ V\times V $ whenever there is more than one irreducible class.

We let $ \Lambda $ be the operator for $ f:V\to \mathbb{R} $
\begin{align*}
	\Lambda f(x) =\frac{1}{\beta(x)}	\sum_{y\in V}\beta(x,y)(f(x)-f(y)).
\end{align*}
Then, the restriction $ \Lambda_{C} $  of $ \Lambda $ to any irreducible class $ C $ gives rise to a symmetric graph Laplacian over the reals, confer \cite{KLW_book}. Hence, by the classical theory of graphs over the reals \cite{Woess00}, \cite{Woess09} or  \cite[Section~6]{KLW_book} there is a Green's function $$  \Gamma_{C} :C\times C\to  (0,\infty] . $$
According to this theory $ \Gamma_{C} $ is finite on all of $ C $ if it is finite on one vertex of $ C $. Furthermore, if $ C $ is not an essential component, then there is are $ x\in C $ and $ y\in V\setminus C $ such that $ x\leadsto y $ in which case we have that $ \Gamma_{C} $ is finite.

\begin{lemma}\label{lem::oneclassgraph}
	Let $b$ be a  graph over $V$ and $ \beta $ the real graph introduced above for $ a\in V $ in an essential component $ C $. Then, 
	\begin{align*}
		\G(a)=\Gamma_{C}(a,a).
	\end{align*}
\end{lemma}
\begin{proof}
Observe that for $ x,y\in V $
\begin{align*}
	\frac{\beta(x,y)}{\beta(x)} =\pi(x,y)=\rho\left(\frac{b(x,y)}{b(x)}\right).
\end{align*}
Since $ C $ is an essential component, we have $ b(x,y)=0 $ for all $ x\in C $ and $ y\in V\setminus C $.

Let $ (K_{n}) $ be an exhaustion of $ C $ and $ a\in K_{0} $. Let $ v_{n} $ be the solution of the Dirichlet problem \eqref{dirpr} for $ K_{n} $ and $ a $.  Since $ 0\preceq v_{n}\preceq 1 $ by the maximum principle, Proposition~\ref{thm::maxpr}, we can define $$  \nu_{n}=\rho (v_{n}) .$$ 
Then, $ \nu_{n} $ satisfies $ \nu_{n}(a) =1$, $ \nu_{n}=0 $ on $ V\setminus K_{n} $ and for $ x\in K_{n} $
\begin{align*}
	\Lambda_{C}\nu_{n}(x) =	\sum_{y\in V} \frac{\beta(x,y)}{\beta(x)}(\nu_{n}(x)-\nu_{n}(y)) = \rho \left(\Delta v_{n}(x)\right).
\end{align*}
Hence, $ \nu_{n} $ solves the Dirichlet problem for $ \Lambda_{C} $ for $ K_{n} \cap C $ and $ a $ over the reals. Thus, by the theory of graphs over the reals, we have $ \Lambda_{C}\nu_{n}(a)>0 $ and as $ (K_{n}\cap C) $ exhausts $ C $, we have
\begin{align*}
	\Gamma_{C}(a)=\lim_{n\to\infty}\frac{1}{\Lambda_{C}\nu_{n}(a)} = \lim_{n\to\infty}\frac{1}{\rho(\Delta v_{n}(a))}= \lim_{n\to\infty}\frac{1}{\rho\left(\frac{\c_{K_{n}}(a)}{b(a)}\right)} = \G(a).
\end{align*}
This finishes the proof.
\end{proof}

Recall that the type of a vertex -- recurrent or transient -- is invariant within the strongly connected component. We come now to the last main result of this paper which characterizes the type of the states in the essential components.

\begin{theorem}\label{thm:charGesscomp}
		Let $b$ be a  graph over $V$,  $ \prob=\prob_{V}^{b} $ and $ C $  an essential component.
		\begin{itemize}
			\item [(a)]  $ C $ is recurrent if and only if $\G(a)=\infty$ for some (all) $a\in C$.
			\item [(b)]  $ C $ is transient if and only if $\G(a)<\infty$ for some (all) $a\in C$.
		\end{itemize}
\end{theorem}
\begin{proof}
	The ``if'' direction follows from Corollary~\ref{cor::rec} and Theorem~\ref{thm:Harnack}.

	For the ``only if'' direction follows from Lemma~\ref{lem::oneclassgraph} and the fact that $ \Gamma_{C}(a,a)=\sum_{n=0}^{\infty} \prob^{(n)}(a,a) $ by the theory of symmetric real graphs since $ C $ is assumed to be an essential component.
\end{proof}

We can summarize the findings of this paper on recurrence and transience in the following final corollary.

\begin{corollary}\label{cor:charGesscomp}
	Let $b$ be a  graph over $V$,  $ \prob=\prob_{V}^{b} $ and $ C $  a strongly connected  component.
	\begin{itemize}
		\item [(a)]  $ C $ is recurrent if and only if $ C $ is essential and  $\G(a)=\infty$ for some (all) $a\in C$.
		\item [(b)]  $ C $ is transient if and only if $ C $ is non-essential or $\G(a)<\infty$ for some (all) $a\in C$.
	\end{itemize}
\end{corollary}
\begin{proof}
	This follows directly from Theorem~\ref{thm:rectranscomponent} and Theorem~\ref{thm:charGesscomp}.
\end{proof}

Of course, Corollary~\ref{cor::rec}  and Theorem~\ref{thm:charGesscomp} give rise to the question if one can characterize recurrence of a state $ a $ in terms of $ \G(a)=\infty $. We demonstrate by examples that this is not the case. First we give an example of a finite graph with a boundary.

\begin{example}\label{ex1}
Let us consider a path graph $ b $ over $ V=\{1,\ldots,5\} $ as on Figure \ref{fig::P52}, with $a=3$ and $U=\{1,2,3,4\}$. The second graph shown in  Figure \ref{fig::P52} is the directed real graph arising from $ b $. By series law \cite[Corollary~6]{Muranova2},  we have
$$
\frac{\c_K(a)}{b(a)}=\dfrac{1}{2(\tau^{-1}+1)}=\dfrac{\tau}{2+2 \tau}
$$
and, hence,
$$
\G_K(a)=\dfrac{1}{\displaystyle\map\left( \frac{\c_K(a)}{b(a)}\right)}=\infty.
$$
On the other hand, $a$ is a transient state which is clear from Figure \ref{fig::P52} as there is no path to return to $ a $.

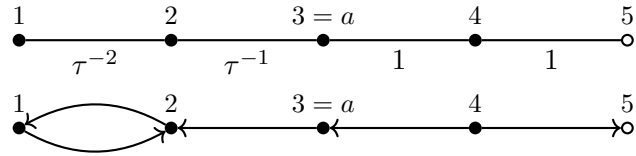
\begin{figure}[H]
\centering
\begin{tikzpicture}[auto,node distance=2cm,
                    thick,main node/.style={circle, draw, fill=black!100,
                        inner sep=0pt, minimum width=4pt}]

  \node[main node] (1) [label={[above]$1$}]{};
  \node[main node] (2) [right of=1,label={[above]$2$}] {};
  \node[main node] (3) [right of=2,label={[above]$3=a$}] {};
  \node[main node] (4) [right of=3,label={[above]$4$}] {};
  \node[main node, fill=none] (5) [right of=4,label={[above]$5$}] {};

  \path[every node/.style={font=\sffamily\small}]
    (2) edge node [] {$\tau^{-2}$} (1)
    (3) edge node [] {$\tau^{-1}$} (2)
    (4) edge node [] {$1$} (3)
    (5) edge node [] {$1$} (4);

\end{tikzpicture}
\begin{tikzpicture}[auto,node distance=2cm,
                    thick,main node/.style={circle, draw, fill=black!100,
                        inner sep=0pt, minimum width=4pt}]

  \node[main node] (1) [label={[above]$1$}]{};
  \node[main node] (2) [right of=1,label={[above]$2$}] {};
  \node[main node] (3) [right of=2,label={[above]$3=a$}] {};
  \node[main node] (4) [right of=3,label={[above]$4$}] {};
  \node[main node, fill=none] (5) [right of=4,label={[above]$5$}] {};

  \path[every node/.style={font=\sffamily\small}]
    (1) edge[bend right, color=black,  ->]  (2)
   (1) edge[bend left, color=black,  <-]  (2)
   (2) edge[color=black,  <-]  (3)
    (3) edge[color=black, <-]   (4)
    (4) edge[color=black,  ->]  (5);

\end{tikzpicture}
\caption{Finite path graph, $a$ is transient, $\G_K(a)=\infty$.}
\label{fig::P52}
\end{figure}
\end{example}

Secondly, we give an example of an infinite graph which shows a transient state can allow for both $ \G(a) $ being finite or infinite.

\begin{example}\label{ex2}
Let us consider two path graphs as at Figure \ref{Fig1inf} with $b(k,k+1)=\tau^{3-k}$ for $k\ge 4$. Let us consider two cases: $b(2,3)=\tau$ and $b(2,3)=\tau^2$. The state $ a $  will be  in both cases transient, see again Figure \ref{Fig1inf}, while 
$$
\G(a)=\dfrac{1}{\displaystyle\map \left(\dfrac{1}{(b(2,3)^{-1}+\tau^{-1}+1+\sum_{i=0}^\infty \tau^i)\tau}\right)}=\begin{cases}
2&: \mbox {  if } b(2,3)=\tau,\\
\infty&: \mbox {  if } b(2,3)=\tau^2.
\end{cases}
$$
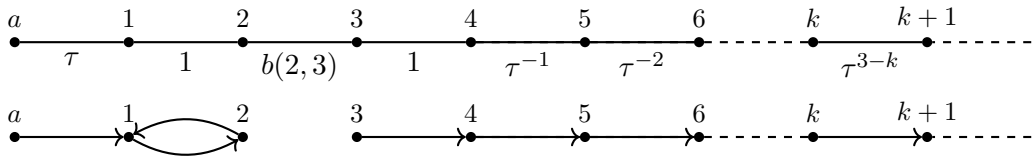
\begin{figure}[H]
\centering
\begin{tikzpicture}[auto,node distance=1.5cm,
                    thick,main node/.style={circle, draw, fill=black!100,
                        inner sep=0pt, minimum width=3pt}]

  \node[main node] (0) [label={[above]$a$}]{};
  \node[main node] (1) [right of=0,label={[above]$1$}]{};
  \node[main node] (2) [right of=1,label={[above]$2$}] {};
  \node[main node] (3) [right of=2,label={[above]$3$}] {};
  \node[main node] (4) [right of=3,label={[above]$4$}] {};
  \node[main node] (5) [right of=4,label={[above]$5$}] {};
  \node[main node] (6) [right of=5,label={[above]$6$}] {};
  \node[main node] (7) [right of=6,label={[above]$k$}] {};
  \node[main node] (8) [right of=7,label={[above]$k+1$}] {};
  \node[draw = none] (9) [right of=8,label={[above]}] {};

  \path[every node/.style={font=\sffamily\small}]
    (1) edge node [bend left] {$\tau$} (0)
    (2) edge node [bend left] {$1$} (1)
    (3) edge node [bend left] {$b(2,3)$} (2)
    (4) edge node [bend left] {$1$} (3)
    (5) edge node [bend left] {$\tau^{-1}$} (4)
    (6) edge node [bend left] {$\tau^{-2}$} (5)
    (8) edge node [bend left] {$\tau^{3-k}$} (7);

 \draw[dashed] (5) to (4);
 \draw[dashed] (6) to (5);
 \draw[dashed] (7) to (6);
 \draw[dashed] (9) to (8);

\end{tikzpicture}
\begin{tikzpicture}[auto,node distance=1.5cm,
                    thick,main node/.style={circle, draw, fill=black!100,
                        inner sep=0pt, minimum width=3pt}]

  \node[main node] (0) [label={[above]$a$}]{};
  \node[main node] (1) [right of=0,label={[above]$1$}]{};
  \node[main node] (2) [right of=1,label={[above]$2$}] {};
  \node[main node] (3) [right of=2,label={[above]$3$}] {};
  \node[main node] (4) [right of=3,label={[above]$4$}] {};
  \node[main node] (5) [right of=4,label={[above]$5$}] {};
  \node[main node] (6) [right of=5,label={[above]$6$}] {};
  \node[main node] (7) [right of=6,label={[above]$k$}] {};
  \node[main node] (8) [right of=7,label={[above]$k+1$}] {};
  \node[draw = none] (9) [right of=8,label={[above]}] {};

  \path[every node/.style={font=\sffamily\small}]
    (0) edge[color=black, ->]  (1)
    (1) edge[bend right, color=black, ->]  (2)
 (1) edge[bend left, color=black, <-]  (2)
    (3) edge[ color=black, ->]  (4)
    (4) edge[color=black, ->]  (5)
    (5) edge[color=black, ->]  (6)
    (7) edge[color=black, ->]  (8);

 \draw[dashed] (5) to (4);
 \draw[dashed] (6) to (5);
 \draw[dashed] (7) to (6);
 \draw[dashed] (9) to (8);

\end{tikzpicture}
\caption{Infinite path graphs with $ \G(a) $ being finite or infinite and $ a $ being transient.}
\label{Fig1inf}
\end{figure}
\end{example}

\end{document}